\documentclass[11pt]{article}
\usepackage{amsmath,amssymb, amsfonts, amsthm, latexsym, hyperref,titlesec,wasysym,commath}
\usepackage{mathrsfs,cleveref}
\usepackage{bbm,xcolor,graphicx}

\titleformat{\subsubsection}[runin] {\normalfont\bfseries}{\thesubsubsection}{0.7em}{\addperiod}
\newcommand{\addperiod}[1]{#1.}

\titlespacing\subsubsection{5pt}{4pt plus 4pt minus 2pt}{5pt plus 2pt minus 2pt}

\newtheorem{theorem}{Theorem}[section]
\newtheorem{lemma}[theorem]{Lemma}

\theoremstyle{remark}
\newtheorem{remark}[theorem]{Remark}

\newcommand{\N}{\mathbb{N}}
\newcommand{\Z}{\mathbb{Z}}
\newcommand{\R}{\mathbb{R}}
\newcommand{\E}{\mathbb{E}}
\renewcommand{\P}{\mathbb{P}}
\newcommand{\eps}{\varepsilon}

\begin{document}
\title{\underline{\textsc{Large, Lengthy Graphs Look Locally Like Lines}}}
\author{\underline{Itai Benjamini and Tom Hutchcroft}}

\AtEndDocument{
  \bigskip
  \small
  \par
  \textsc{Itai Benjamini, Weizmann Institute of Science} \par
  \textit{Email:} \texttt{itai.benjamini@weizmann.ac.il}\par

\medskip

  \textsc{Tom Hutchcroft, Statslab, DPMMS, University of Cambridge} \par
  \textit{Email:} \texttt{t.hutchcroft@maths.cam.ac.uk} \par

}

\date{\underline{May 9, 2019}}

\maketitle

\begin{abstract}
We apply the theory of unimodular random rooted graphs to study the metric geometry of large, finite, bounded degree graphs whose diameter is proportional to their volume. We prove that for a positive proportion of the vertices of such a graph, there exists a mesoscopic scale on which the graph looks like $\R$ in the sense that the rescaled ball is close to a line segment in the Gromov-Hausdorff metric.
\end{abstract}


\section{{Introduction}}

The aim of this modest note is to prove that large graphs with diameter proportional to their volume must `look like $\R$' from the perspective of a positive proportion of their vertices, after some rescaling that may depend on the choice of vertex. We write $d^\mathrm{loc}_\mathrm{GH}$ for the \emph{local Gromov-Hausdorff metric}, a measure of similarity between locally compact pointed metric spaces that we define in detail in Section~\ref{sec:GHbackground}.

\begin{theorem}
\label{thm:GH}
Let $(G_n)_{n\geq 1}=((V_n,E_n))_{n\geq 1}$ be a sequence of finite, connected graphs with $|V_n|\to\infty$, and  suppose that there exists a constant $C<\infty$ such that $|V_n| \leq C \operatorname{diam}(G_n)$ 
for every $n\geq 1$. Suppose furthermore that the set of degree distributions of the graphs $G_n$ are uniformly integrable. Then there exists a 
 sequence of subsets $A_n \subseteq V_n$ with $\liminf_{n\to\infty}|A_n|/|V_n| \geq C^{-1}$ such that
\[
\lim_{n\to\infty} \sup_{v\in A_n} \inf_{\eps>0} d^\mathrm{loc}_\mathrm{GH}\biggl(\Bigl(V_n,\, \eps d_{G_n},\, v \Bigr),\, \Bigl(\R,\, d_\R,\, 0 \Bigr)\biggr) =0,
\]
where we write $d_{G_n}$ for the graph metric on $G_n$ and $d_\R(x,y)=|x-y|$ for the usual metric on $\R$.
\end{theorem}

Here, the degree distributions of the graphs $(G_n)_{n\geq 1}$ are said to be \textbf{uniformly integrable} if for every $\eps>0$ there exists $M$ such that $\sum_{v\in V_n} \deg(v) \mathbbm{1}(\deg(v) > M) \leq \eps |V_n|$ for every $n\geq 1$.
This hypothesis holds in particular if there exists a constant $M<\infty$ such that the graphs $(G_n)_{n\geq 1}$ all have degrees bounded $M$.

\medskip

 Our result should be compared to the (much more difficult) result of the first author, Finucane, and Tessera \cite{BFT} that \emph{transitive} graphs with volume proportional to their diameter converge to the circle when rescaled by their diameter; here we have much weaker hypotheses but also a much weaker result. We remark that in \cite{BFT} it is shown more generally that every sequence of transitive graphs with volume at most \emph{polynomial} in their diameter has a subsequence converging to a torus (equipped with an \emph{invariant Finsler metric}) when rescaled by their diameters. Various further results on the scaling limits of transitive graphs satisfying polynomial growth conditions have subsequently been obtained by Tessera and Tointon \cite{tessera2017scaling}. It seems unlikely that polynomial growth assumptions such as $|V_n| = O\bigl( \operatorname{diam}(G_n)^C\bigr)$ will imply anything of comparable strength to Theorem~\ref{thm:GH} about the metric geometry of graphs without the assumption of transitivity.  
 
 \medskip

\begin{figure}
\includegraphics[width=\textwidth]{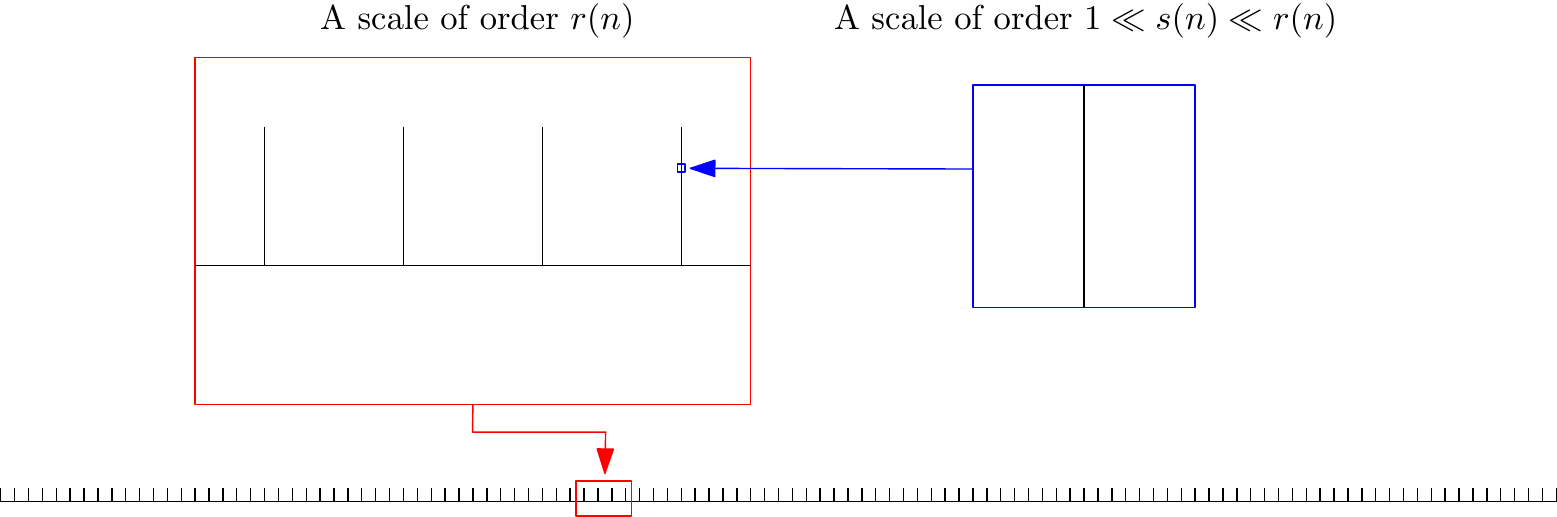}
\caption{A `comb' of length $n$ with regularly spaced `teeth' of length $r(n)$ with $1 \ll r(n)\ll n$. On the scale $r(n)$, the comb is metrically distinguishable from $r(n)$ at every point. On scales $s(n)$ satisfying $r(n) \ll s(n) \ll n$, the teeth of the comb are negligible and the comb looks like $\R$ from most of its points, with the exception of points near the left or right ends of the comb where the comb looks locally like $\R_+=[0,\infty)$. On scales $s(n)$ satisfying $1 \ll s(n) \ll r(n)$, the comb again looks like $\R$ from the perspective of most of its points, with the only exceptions being the points that are close to the top or bottom of a tooth (these bad points can be either on the teeth or the body of the comb).}
\label{fig:comb}
\end{figure}

 Note that it is not possible to control the scale on which the rescaled graph looks like $\R$. Indeed, if $r:\N\to \N$ is any function with $r(n)/n \to 0$ and $r(n)\to\infty$ as $n\to\infty$, then the graph formed by attaching $\lceil n/r(n) \rceil$ line segments of length $ \lfloor r(n)\rfloor $ to a line segment of length $n$ in a regularly spaced manner has $\operatorname{diam}(G_n) \sim n$ and $|V_n| \sim 2n$ but it metrically distinguishable from $\R$ on the scale $r(n)$ from the perspective of every vertex in the sense that
\[
\lim_{n\to\infty}\inf_{v\in V} d^\mathrm{loc}_\mathrm{GH}\biggl(\Bigl(V_n,\, r(n)^{-1} d_{G_n},\, v \Bigr),\, \Bigl(\R,\, d_\R,\, 0 \Bigr)\biggr) >0.
\]
See Figure~\ref{fig:comb}.
Note also that the hypotheses do \emph{not} allow us to take the sets $A_n$ to have $|A_n|/|V_n| \to 1$. Indeed, consider taking an $n \times n$ square grid and attaching to this grid a path of length $n^2$: The volume of this graph is about twice its diameter, the grid has about $1/2$ of the total vertices of the graph, and from a vertex of the grid the graph is metrically distinguishable from  $\R$ at every scale. See Figure~\ref{fig:gridline}. (In this example, the graph instead looks locally like $\R_+$ from the perspective of the vertices of the grid. One can construct similar examples where the graph does not look locally like either $\R$ or $\R_+$ by attaching, say, three disjoint paths to the grid instead of one.) A similar example shows that the dependence $\liminf_{n\to\infty} |A_n| / |V_n| \geq C^{-1}$ is optimal. 

\begin{figure}
\centering
\includegraphics[width=0.75\textwidth]{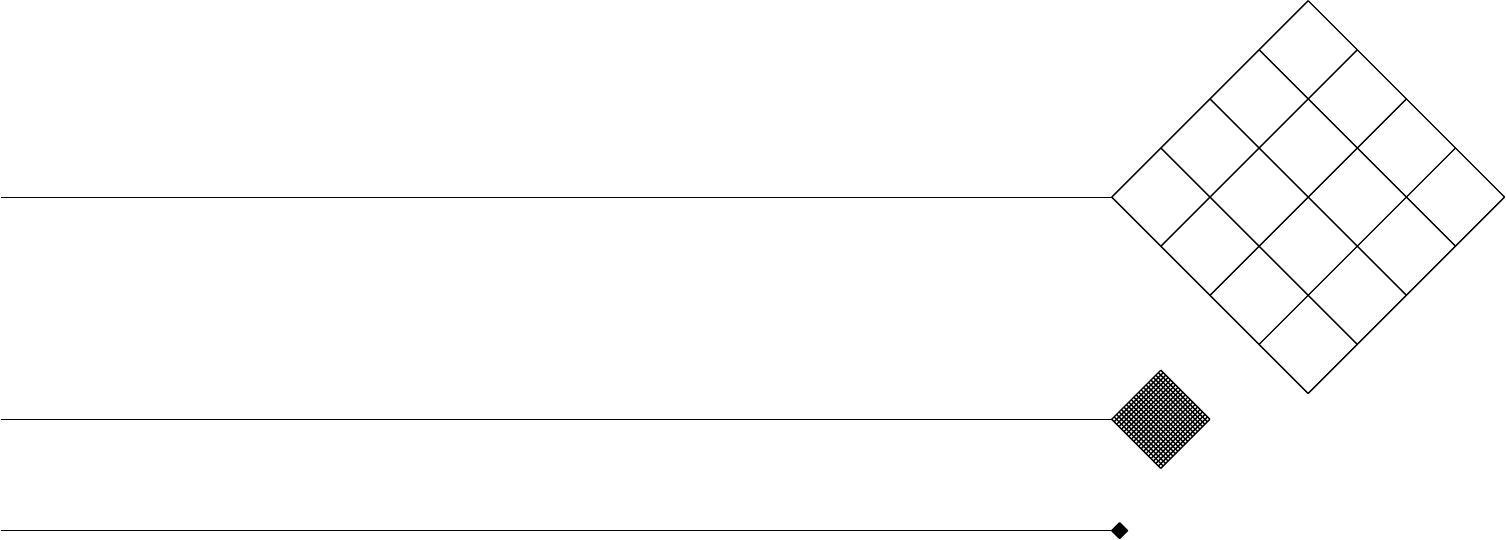}
\caption{The graphs formed by attaching a line of length $n^2$ to the corner of an $n\times n$ square grid, where $n=5,20,100$. When $n$ is large, the $n \times n$ grid is negligible from the perspective of the metric at the scales much larger than $n$ but has a constant proportion of the total volume. If we pick a point in the $n \times n$ grid uniformly at random, the graph will  look locally like $(\R^2,\|\cdot\|_1)$ on scales $1 \ll s(n) \ll n$ and like $\R_+$ on scales $n \ll s(n) \ll n^2$ with high probability. On scales of order $n$, the graph will look like a metric space formed by attaching a copy of $\R_+$ to the corner of $([0,1],\lambda \|\cdot\|_1)$, where $\lambda \in (0,\infty)$ depends on the precise choice of scale. From the perspective of a uniform random point in the line, which will not be near the ends of the line with high probability, the graph will look locally like $\R$ on all scales $1 \ll s(n) \ll n^2$.}
\label{fig:gridline}
\end{figure}

\medskip

\noindent
\textbf{Unimodular random rooted graphs and Benjamini-Schramm limits.}
Although the statement is not probabilistic, the proof of Theorem~\ref{thm:GH} will rely on probabilistic arguments in an essential way. Roughly speaking, we will use the theory of \emph{Benjamini-Schramm limits} to reduce Theorem~\ref{thm:GH} to a statement about certain probabilistic objects arising as subsequential limits in distribution of the sequence $(G_n,\rho_n)$, where $\rho_n$ is a uniform random vertex of $G_n$ for each $n\geq 1$. 


Let us now recall the relevant definitions, referring the reader to \cite{CurienNotes,AL07} for more detailed treatments. 
A \textbf{rooted graph} is a connected, locally finite graph $g$ together with a distinguished root vertex $v$. A graph isomorphism between rooted graphs is an isomorphism of rooted graphs if it preserves the root. The space of (isomorphism classes of) rooted graphs is denoted $\mathcal{G}_\bullet$, and is equipped with the \textbf{local topology}, which is induced by the metric
\[
d_\mathrm{loc}((g_1,v_1),(g_2,v_2)) = 2^{-R((g_1,v_1),(g_2,v_2))} 
\]
where $R((g_1,v_1),(g_2,v_2))$ is the supremal value of $r$ for which the balls of radius $r$ around $v_1$ and $v_2$ are isomorphic as rooted graphs. 
 The space $\mathcal{G}_\bullet$ is a Polish space (i.e., a topological space with a compatible complete metric) with respect to this topology \cite[Theorem 2]{CurienNotes}, and moreover is separable since the (isomorphism classes of) finite rooted graphs form a countable dense subset. We will typically ignore the distinction between a rooted graph and its isomorphism class when this does not cause problems.

Similarly, we define a \textbf{rooted oriented-edge-labelled} graph to be a rooted graph $(g,v)$ together with a function from the set of oriented edges of $g$ to $\{0,1\}$. (Although our graphs are undirected, we can still think of each edge as a pair of oriented edges.) We write $\mathcal{G}_\bullet^{\{0,1\}}$ for the space of isomorphism classes of rooted oriented-edge-labelled graphs, which is equipped with a local topology that is defined similarly to the unlabelled case. 
Finally, we define the spaces of \textbf{doubly-rooted graphs} $\mathcal{G}_{\bullet\bullet}$ and  
\textbf{doubly-rooted oriented-edge-labelled graphs} $\mathcal{G}^{\{0,1\}}_{\bullet\bullet}$  similarly to above except that we now have an \emph{ordered pair} of distinguished root vertices.

A probability measure $\mu$ on $\mathcal{G}_\bullet$ is said to be \textbf{unimodular} if it satisfies the \textbf{mass-transport principle}, which states that the identity\footnote{Here and elsewhere, given a random variable $X$ with law $\mu$ taking values in a set $\Omega$ and a measurable function $f:\Omega\to [0,\infty]$, we write $\mu[f(X)]$ for the expectation of the random variable $f(X)$.}
\[
\mu\left[ \sum_{v\in V} F(G,\rho,v) \right] = \mu\left[ \sum_{v\in V} F(G,v,\rho) \right]
\]
is satisfied for every measurable function $F : \mathcal{G}_{\bullet \bullet} \to [0,\infty]$, where we write $(G,\rho)$ for a random variable sampled from the measure $\mu$. Unimodular probability measures on $\mathcal{G}_\bullet^{\{0,1\}}$ are defined similarly. We think of the function $F$ as being an `automorphism equivariant' rule for sending non-negative amounts of mass between the different vertices of a graph $g$: The mass-transport principle says that for any such rule, the expected total mass the root sends out is equal to the expected total mass that the root receives. 
This notion was first introduced by Benjamini and Schramm~\cite{BeSc} and developed systematically by Aldous and Lyons \cite{AL07}. A related form of the mass-transport principle was first considered by H\"aggstr\"om in the context of Cayley graphs \cite{Haggstrom97}.

The set of unimodular probability measures on $\mathcal{G}_\bullet$ is convex and closed under the weak topology \cite[Theorem 8]{CurienNotes}. It also includes all the laws of random graphs of the form $(G,\rho)$ where $G$ is a deterministic finite connected graph and $\rho$ is a uniform random root vertex of $G$. Indeed, if we denote the law of this random rooted graph by $\mu$ and take $F:\mathcal{G}_\bullet\to[0,\infty]$ to be measurable then
\[
\mu\!\left[ \sum_{v\in V} F(G,\rho,v) \right] =
\mu\!\left[ \frac{1}{|V|}\sum_{u\in V} \sum_{v\in V} F(G,u,v) \right]
=
\mu\!\left[ \frac{1}{|V|}\sum_{u\in V} \sum_{v\in V} F(G,v,u) \right]
=
\mu\!\left[ \sum_{v\in V} F(G,v,\rho) \right]
\]
as required. Intuitively, we think of a random rooted graph $(G,\rho)$ as being unimodular if the root is `uniformly distributed on the vertex set', even though this does not make literal sense when $G$ is infinite. 
 Similar statements hold for rooted oriented-edge-labelled graphs. 

 A unimodular random rooted graph $(G,\rho)$ is said to be the \textbf{Benjamini-Schramm} limit of a sequence of finite connected graphs $(G_n)_{n\geq 1}$ if the random variables $(G_n,\rho_n)$ converge in distribution to $(G,\rho)$ when we take $\rho_n$ to be a uniform random root vertex of $G_n$ for each $n\geq 1$. Concretely, this means that for each fixed $r\geq 1$ and each fixed rooted graph $(g,u)$ we have that
 \[
\frac{1}{|V_n|} \sum_{v\in V_n} \mathbbm{1}\bigl(B_r(G_n,v)
 = B_r(g,u)\bigr)
= \P\bigl(B_r(G_n,\rho_n) = B_r(g,u)\bigr)
  \xrightarrow[n\to\infty]{} \P\bigl(B_r(G,\rho) = B_r(g,u)\bigr),
 \]
 where $B_r(g,u)$ denotes the ball of radius $r$ around $u$ in $g$, considered as a rooted graph.

The uniform integrability hypothesis in Theorem~\ref{thm:GH} is used as a compactness assumption. 
Let us now recall some of the relevant basic theory of weak convergence of probability measures. Let $\mathbb{X}$ be a metric space and let $\mathcal{P}(\mathbb{X})$ be the space of probability measures on $\mathbb{X}$, which we equip with the weak topology, i.e., the weakest topology making the maps $\mu \mapsto \mu(f)$ continuous for every bounded continuous function $f:\mathbb{X}\to \R$. (The weak topology on measures is usually called the weak-$*$ topology outside the probabilistic context.) When $\mathbb{X}$ is separable, the weak topology on $\mathcal{P}(\mathbb{X})$ is metrisable \cite[Remark 13.14.ii]{Klenkebook}. 
 Recall that weak convergence of measures and convergence in distribution of random variables are essentially the same thing: A sequence of random variables $(X_n)_{n\geq 1}$ in $\mathbb{X}$ converges in distribution to some random variable $X$ in $\mathbb{X}$ if and only if the laws $(\mu_n)_{n\geq 1}$ converge weakly to the law $\mu$ of $X$.
 Prohorov's theorem \cite[Theorem 13.29]{Klenkebook} states that if $\mathbb{X}$ is Polish then a family of probability measures $\{\mu_i :i\in I\}$ on $\mathbb{X}$ is precompact in $\mathcal{P}(\mathbb{X})$ if and only if it is \textbf{tight}, meaning that for every $\eps>0$ there exists a compact set $K \subseteq \mathbb{X}$ such that $\mu_i(\mathbb{X} \setminus K) \leq 1-\eps$ for every $i\in I$. A family of \emph{random variables} is said to be tight if the associated family of probability measures is tight.

 In our context, it is a theorem of Curien \cite[Proposition 21 and Exercise 27]{CurienNotes} that uniform integrability of the degree of the root is a sufficient (but not necessary) condition for a family of unimodular probability measures on $\mathcal{G}_\bullet$ to be tight in $\mathcal{P}(\mathcal{G}_\bullet)$. 
 Thus, the uniform integrability hypothesis of Theorem~\ref{thm:GH} ensures that if $\rho_n$ is a uniform random root vertex of $G_n$ for each $n\geq 1$ then the sequence of random rooted graphs $(G_n,\rho_n)$ is tight with respect to the local topology. Since $\mathcal{P}(\mathcal{G}_\bullet)$ is metrisable, compactness implies sequential compactness and we deduce that there exists a subsequence $\sigma(n)$ such that $G_{\sigma(n)}$ Benjamini-Schramm converges to some infinite unimodular random rooted graph $(G,\rho)$. We will use this fact to deduce Theorem~\ref{thm:GH} from the following theorem via standard compactness arguments. 

\begin{theorem}
\label{thm:BStwoends}
Let $(G_n)_{n\geq 1}=((V_n,E_n))_{n\geq 1}$ be a sequence of finite, connected graphs Benjamini-Schramm converging to some infinite random rooted graph $(G,\rho)=((V,E),\rho)$. Suppose that there exists a constant $C<\infty$ such that
$|V_n| \leq C \operatorname{diam}(G_n)$ 
for every $n\geq 1$. Then there exists an event $\Omega$ of probability at least $1/C$ on which the following hold:
\begin{enumerate}
	\item The graph $G$ is two-ended and has linear growth almost surely.
	\item 
The pointed metric space $(V,\varepsilon d_G, \rho)$ converges to $(\R,d_\R,0)$ in the local Gromov-Hausdorff topology almost surely as $\eps \downarrow 0$. 
\end{enumerate}
\end{theorem}

We remark that only item 2 of this theorem is used in the proof of Theorem~\ref{thm:GH}; item 1 is included as it follows from essentially the same proof and is of independent interest.
Here, a connected graph $G$ is said to have \textbf{linear growth} if $\limsup_{n\to\infty} \frac{1}{n}|B_r(G,v)| < \infty$ for some (and hence every) vertex $v$ of $G$. An infinite, connected, locally finite graph is said to be \textbf{$k$-ended} if deleting a finite set of vertices from $G$ results in a maximum of $k$ infinite connected components. 

Similarly to above, the hypotheses of Theorem~\ref{thm:BStwoends} do \emph{not} ensure that $G$ is two-ended or has linear growth almost surely. Consider, for example, taking $G_n$ to be a path of length $n^2$ connected to the corner of an $n \times n$ square grid as in Figure~\ref{fig:gridline}. The Benjamini-Schramm limit of this sequence is the random graph that is equal either to $\Z$ or $\Z^2$, each with probability $1/2$. Indeed, if the uniform random root vertex belongs to the line then it is not close to the boundary of the line with high probability, so that the graph looks locally like $\Z$ from its perspective. Similarly, it the uniform random root vertex belongs to the grid then it is not close to the boundary of the grid with high probability, so that the graph looks locally like $\Z^2$ from its perspective.


\subsection{{The Gromov-Hausdorff metric}}
\label{sec:GHbackground}

We now define the Gromov-Hausdorff metric, referring the reader to \cite{BI} for a detailed treatment of this metric and its properties. Given two sets $X$ and $Y$, a \textbf{correspondence} between $X$ and $Y$ is a set $\mathscr{R} \subseteq X \times Y$ such that $\{x\} \times Y \cap \mathscr{R} \neq \emptyset$ and $X \times \{y\} \cap \mathscr{R} \neq \emptyset$ for every $x\in X$ and $y\in Y$. If $(X,x_0)$ and $(Y,y_0)$  are \textbf{pointed sets}, i.e., non-empty sets each with a distinguished point, then we say that a correspondence $\mathscr{R}$ between $X$ and $Y$ is a correspondence between $(X,x_0)$ and $(Y,y_0)$ if $(x_0,y_0) \in \mathscr{R}$. If $(X,d_X)$ and $(Y,d_Y)$ are metric spaces and $\mathscr{R}$ is a correspondence between $X$ and $Y$, we define the \textbf{distortion} of $\mathscr{R}$ to be
\[
\operatorname{dis}\mathscr{R} = \sup\Bigl\{\bigl|d_X(x,x')-d_Y(y,y')\bigr| : (x,y),(x',y') \in \mathscr{R}\Bigr\}.
\]
Given two pointed metric spaces $(X,d_x,x_0)$ and $(Y,d_Y,y_0)$, we define the \textbf{Gromov-Hausdorff} distance to be
\[
d_\mathrm{GH}\bigl((X,d_X,x_0),(Y,d_Y,y_0)\bigr) = \frac{1}{2} \inf \Bigl\{ \operatorname{dis}\mathscr{R} : \mathscr{R} \text{ is a correspondence between $(X,x_0)$ and $(Y,y_0)$}\Bigr\}.
\]
The function $d_\mathrm{GH}$ defines a metric on the space of \emph{isometry classes} of compact pointed metric spaces.
Similarly, the \textbf{local Gromov-Hausdorff topology} on (isometry classes of) \emph{locally compact} pointed metric spaces is defined to be the topology induced by the metric
\begin{equation*}
d_\mathrm{GH}^\mathrm{loc} \bigl((X,d_X,x_0),(Y,d_Y,y_0)\bigr)\\
= \sum_{r\geq 1} 2^{-r} d_\mathrm{GH}\Bigl(\bigl(B_X(x_0,r),d_X,x_0\bigr),\bigl(B_Y(y_0,r),d_Y,y_0\bigr)\Bigr) 
\end{equation*}
where we write $B_X(x,r)$ for the ball of radius $r$ around the point $x$ in the metric space $X=(X,d_X)$.
This topology has the property that $(X_n,d_n,x_n)$ converges to $(X,d_X,x_0)$ if and only if 
\[
\lim_{n\to\infty}
d_\mathrm{GH}\Bigl(\bigl(B_{X_n}(x_n,r),d_n,x_n\bigr),\bigl(B_X(x_0,r),d_X,x_0\bigr)\Bigr) =0\]
for every $r\geq 0$. Note that there is a close analogy between the local Gromov-Hausdorff topology on the set of pointed metric spaces and the local topology on the set of rooted graphs.

\section{{Proof}}


\noindent \textbf{The portmanteau theorem.} Before starting the proof, let us recall a basic and classical fact about weak convergence of probability measures. This fact is often grouped together with several other related facts and known as the \emph{portmanteau theorem} \cite[Theorem 13.16]{Klenkebook}. Recall that if $\mathbb{X}$ is a metric space, a function $f:\mathbb{X} \to \R$ is said to be 
\textbf{lower semi-continuous} if $\liminf_{n\to\infty} f(x_n) \geq f(x)$ whenever $x_n \to x$, and similarly that 
$f$ is said to be \textbf{upper semi-continuous} if $\limsup_{n\to\infty} f(x_n) \leq f(x)$ whenever $x_n \to x$. 


\begin{theorem}
Let $\mathbb{X}$ be a metric space, and let $\mathcal{P}(\mathbb{X})$ be the space of probability measures on $\mathbb{X}$. 
\begin{enumerate}

  \item If $V \subseteq \mathbb{X}$ is closed, the map $\mu \mapsto \mu(V)$ is a lower semi-continuous function on $\mathcal{P}(\mathbb{X})$.
  \item If $U \subseteq \mathbb{X}$ is open, the map $\mu \mapsto \mu(U)$ is an upper semi-continuous function on $\mathcal{P}(\mathbb{X})$.
\end{enumerate}
\end{theorem}



We will also use the following classical fact about stationary sequences of random variables. Recall that if $(X_n)_{n\in \Z}$ is a doubly-infinite sequence of random variables defined on a common probability space, we say that $(X_n)_{n\in \Z}$ is \textbf{stationary} if $(X_{n+k})_{n\in \Z}$ has the same distribution as $(X_n)_{n\in \Z}$ for every $k\in \Z$.

\begin{lemma}
\label{lem:ergodic}
Let $(X_n)_{n\in \Z}$ be a stationary sequence of $[0,\infty)$-valued random variables such that $\E X_0 < \infty$. Then \begin{align}
\label{eq:mass1}
\limsup_{n\to\infty} \frac{1}{2n+1}\sum_{m=-n}^n X_m &< \infty && \text{ and }\\
\limsup_{n\to\infty} \frac{1}{2n+1}\max\Bigl\{ X_m : -n \leq m \leq n \Bigr\} &=0
\label{eq:max1}
\end{align}
almost surely.
\end{lemma}

\begin{proof}
Birkhoff's pointwise ergodic theorem \cite[Theorem 20.14]{Klenkebook} implies that $\frac{1}{2n+1}\sum_{m=-n}^n X_m$ converges almost surely to the conditional expectation of $X_0$ given the invariant sigma-algebra, which is almost surely finite as required. For the second claim, we have by stationarity and a union bound that
\[\sum_{m\geq 1} \P\bigl(\max\{X_m, X_{-m}\} \geq \eps m\bigr) \leq 2 \sum_{m\geq 1} \P(X_0 \geq \eps m) \leq 2\int_0^\infty \P(X_0 \geq \eps t) \dif t = \frac{2}{\eps} \E X_0 < \infty \]
for each $\eps>0$. It follows by Borel-Cantelli that the event $\{\max\{X_m, X_{-m}\} \geq \eps m\}$ occurs for at most finitely many $m\geq 1$ almost surely for each $\eps>0$, and the claim (2) follows easily.
\end{proof}


We are now ready to begin the proofs of our main theorems.

\begin{proof}[Proof of Theorem \ref{thm:BStwoends}]
Let $(G_n,\rho_n)$ and $(G,\rho)$ be as in the statement of the theorem. 
Let $E_n^\rightarrow$ and $E^\rightarrow$ denote the sets of oriented edges of $G_n$ and $G$ respectively. 
We first argue that (by passing to a bigger probability space if necessary) it is possible to endow $G$ with a random oriented-edge-labelling $\gamma \in \{0,1\}^{E^\rightarrow}$ such that the following hold:
\begin{enumerate}
\item[(1)] $(G,\rho,\gamma)$ is a unimodular random rooted oriented-edge-labelled graph.
\item[(2)] The event $\Omega:=\{ \gamma(e)=1 \text{ for some $e\in E^\rightarrow$}\}$ has $\P(\Omega) \geq C^{-1}$.
\item[(3)] On the event $\Omega$, the set $\{e \in E^\rightarrow : \gamma(e)=1\}$ is an oriented doubly-infinite geodesic of $G$ almost surely. 
\end{enumerate} 
We will construct the edge-labelling $\gamma$ via a limiting procedure. 
For each $n \geq 1$, let $\gamma_n$ be an oriented geodesic of maximum length in $G_n$, which we consider as an element of $\{0,1\}^{E_n^\rightarrow}$. Observe that the map $\pi: \mathcal{G}_\bullet^{\{0,1\}} \to \mathcal{G}_\bullet$ defined by forgetting the labelling is continuous and proper, meaning that the preimage of any compact set is compact. It follows that the sequence of oriented-edge-labelled graphs $(G_n,\rho_n,\gamma_n)$ is tight, and hence that there exists a subsequence $(G_{\sigma(n)},\rho_{\sigma(n)},\gamma_{\sigma(n)})$ converging in distribution to some infinite unimodular random edge-labelled graph $(G,\rho,\gamma)$. The notation here is justified since, by continuity of $\pi$, forgetting the oriented-edge-labelling gives back the same law on random rooted graphs that described our original random rooted graph $(G,\rho)$.

 It remains to argue that this graph satisfies properties (2) and (3) above. 
We say that a vertex $v$ is \textbf{incident} to $\omega \in \{0,1\}^{E^\rightarrow}$ if there is an oriented edge $e$ with $v$ as one of its endpoints and with $\omega(e)=1$. Since $\gamma_n$ is a geodesic of maximal length it is incident to exactly $\operatorname{diam}(G_n)+1$ vertices, and since $\rho_n$ is uniform on $V_n$ it follows that
\[
\P(\rho_n \text{ is incident to an edge of $\gamma_n$}) = \frac{\operatorname{diam}(G_n) +1}{|V_n|} \geq C^{-1}.
\]
Taking the limit as $n\to\infty$, it follows that
\[
\P(\gamma(e)=1 \text{ for some $e\in E^\rightarrow$}) \geq \P(\rho \text{ is incident to an edge of $\gamma$}) \geq C^{-1}
\] 
also, establishing the property (2).  For (3), we consider the set  $A$ of all 
$(g,x,\omega) \in \mathcal{G}_\bullet^{\{0,1\}}$ such that $\omega$ does not contain any oriented cycles, any two vertices incident to $\omega$ are connected in $\omega$ by exactly one path, and that this path is an oriented geodesic in $g$.
We observe that $A$
is closed in $\mathcal{G}_\bullet^{\{0,1\}}$ and that $(G_n,\rho_n,\gamma_n)\in A$ almost surely for every $n\geq 1$. Applying the portmanteau theorem, we deduce that $(G,\rho,\gamma)\in A$ almost surely also, and hence that the  set $\{e \in E^\rightarrow : \gamma(e)=1\}$ is almost surely an oriented  geodesic of $G$ on the event that it is nonempty. Moreover, since $G$ is infinite, the mass-transport principle implies that $\gamma$ must be a doubly-infinite geodesic on this event. If not, there would exist one or two special vertices of $(G,\rho,\gamma)$ that lied at the endpoints of $\gamma$. Applying the mass-transport principle to the function
\[F(g,u,v,\omega):=\mathbbm{1}(v \text{ is incident to exactly one edge of $\omega$})\]
 would then lead to a contradiction: Indeed, the expected mass sent out by the root would be at most $2$, while the root would receive infinite mass with positive probability by \cite[Proposition 11]{CurienNotes} and would therefore receive infinite mass in expectation.

Next, recall that a set $A \subseteq \mathcal{G}_\bullet^{\{0,1\}}$ is said to be  \textbf{re-rooting invariant} if $(g,v) \in A$ if and only if $(g,u) \in A$ for every vertex $u$ of $g$. It is easily seen that if $\mu$ is a unimodular probability measure on $\mathcal{G}_\bullet^{\{0,1\}}$ and $A$ is a measurable, re-rooting invariant set with $\mu(A)>0$ then the conditional measure $\mu( \,\cdot \mid A)$ is also unimodular. (The unfamiliar reader may find the proof of this statement to be a simple but illuminating exercise on the mass-transport principle. 
We note moreover that it is a theorem of Aldous and Lyons \cite[Theorem 4.7]{AL07} that a unimodular probability measure on $\mathcal{G}_\bullet$ is an extreme point of the set of unimodular probability measures if and only if it gives probability either zero or one to every re-rooting invariant event.)
 In our context, since the event $\Omega$ is re-rooting invariant, we deduce that if $\mu$ denotes the conditional law of $(G,\rho,\gamma)$ given $\Omega$ then $\mu$ is unimodular.

We now use the random oriented doubly-infinite geodesic $\gamma$ to argue that both claims of the theorem hold on the event $\Omega$.
 Let $\Gamma$ be the set of vertices visited by $\gamma$.
Since $\gamma$ is oriented, we can put a total ordering on $\Gamma$ that encodes the order in which the vertices of $\Gamma$ are visited by $\gamma$. For each $v\in \Gamma$, we write $(\sigma^n(v))_{n \in \Z}$ for the vertices obtained by shifting up and down the oriented geodesic $\gamma$, starting with $\sigma^0(v)=v$.
 For each vertex $v$ of $G$, let $g(v)$ be a point of $\Gamma$ at minimal distance to $v$, choosing $g(v)$ to be the point that is minimal in the total order on $\Gamma$ 
if there are multiple points at minimal distance to $v$.
The mass-transport principle implies that
\[
\mu\left[\#\{v \in V: g(v)=\rho\} \right] = \mu\left[\sum_{v\in V} \mathbbm{1}(g(v)=\rho) \right]  = \mu\left[\sum_{v\in V} \mathbbm{1}(g(\rho)=v) \right] = 1
\]
and hence that
\begin{equation}
\label{eq:finite_expectation}
\mu\left[\#\{v \in V: g(v)=\rho\} \mid \rho \in \Gamma \right] = \mu(\rho \in \Gamma)^{-1}<\infty.
\end{equation}
Let $X_0=g(\rho)$ and let the sequence $(X_n)_{n \in \Z}=(\sigma^n(X_0))_{n\in\Z}=(\sigma^n(g(\rho)))_{n\in \Z}$ be defined by shifting up and down the oriented geodesic $\gamma$. 
%
Let $\tilde \mu$ be the law of $(G,\rho,\gamma)$ conditioned on the event that $\rho \in \Gamma$. The mass-transport principle implies that the sequence of random variables $(G,X_n,\gamma)_{n \in \Z}$ is stationary under the measure $\tilde \mu$: Indeed, if $\mathscr{A} \subseteq \bigl(\mathcal{G}_\bullet^{\{0,1\}}\bigr)^\Z$ is any measurable set and $k\in \Z$ then
\begin{align*}
\tilde \mu\left(\left((G,X_{n+k},\gamma)\right)_{n\in \Z} \in \mathscr{A}  \right) &= \mu(\rho \in \Gamma)^{-1}\mu\left[\sum_{v\in V} \mathbbm{1}\left(\rho \in \Gamma, v = \sigma^k(\rho), ((G,\sigma^n(v),\gamma))_{n\in \Z} \in \mathscr{A}\right) \right]
\\&= \mu(\rho \in \Gamma)^{-1}\mu\left[\sum_{v\in V} \mathbbm{1}\left(v \in \Gamma, \rho = \sigma^k(v), ((G,\sigma^n(\rho),\gamma))_{n\in \Z} \in \mathscr{A} \right) \right]\\
&= \tilde \mu\left(\left((G,X_{n},\gamma)\right)_{n\in \Z} \in \mathscr{A}  \right),
\end{align*}
which establishes the desired stationarity. Setting $K_n = \{ v \in V : g(v)=X_n\}$, it follows from this and \eqref{eq:finite_expectation} that
$(|K_n|)_{n \in \Z}$
is a stationary sequence of finite mean random variables under the measure $\tilde \mu$, and  we deduce from Lemma~\ref{lem:ergodic} that
\begin{align}
\label{eq:mass}
\limsup_{n\to\infty} \frac{1}{2n+1}\sum_{m=-n}^n |K_m| &< \infty && \text{ and }\\
\limsup_{n\to\infty} \frac{1}{2n+1}\max\Bigl\{ |K_m| : -n \leq m \leq n \Bigr\} &=0
\label{eq:max}
\end{align}
almost surely under $\tilde \mu$. 
On the other hand, we have that
\begin{align*}
\mu( ((G,X_n,\gamma))_{n\in \Z} \in \mathscr{A} ) &= 
\mu\left[ \sum_{v\in V} \mathbbm{1}\left(g(\rho)=v, ((G,\sigma^n(v),\gamma))_{n\in \Z} \in \mathscr{A} \right)\right]\\
&=
\mu\left[ \sum_{v\in V} \mathbbm{1}\left(g(v)=\rho, ((G,\sigma^n(\rho),\gamma))_{n\in \Z} \in \mathscr{A} \right)\right]\\
&=  \mu(\rho \in \Gamma) \tilde\mu\left[ \#\{v \in V : g(v) =\rho\} \mathbbm{1}(((G,X_n,\gamma))_{n\in \Z} \in \mathscr{A})\right]
\end{align*}
for every measurable set $\mathscr{A} \subseteq (\mathcal{G}_\bullet^{\{0,1\}})^\Z$, so that the laws of $((G,X_n,\gamma))_{n\in \Z}$ under $\mu$ and $\tilde \mu$ are absolutely continuous and hence that \eqref{eq:mass} and \eqref{eq:max} also hold almost surely under $\mu$.

We will now argue that this implies the two claims. We begin with the first. Linear growth follows obviously from \eqref{eq:mass} since the ball of radius $n$ around $X_0$ is contained in the set $\bigcup_{m=-n}^n K_m$. Since $G$ has linear growth it must have at most two-ends, since unimodular random rooted graphs with more than two ends always have infinitely many ends and exponential growth almost surely \cite[Theorem 8.13]{AL07}. (Note that we do not need this result for the proof of Theorem~\ref{thm:GH}.) To see that $G$ is two-ended rather than one-ended, observe that 
if $K_n$ and $K_m$ are adjacent then we must have that 
$\operatorname{diam}(K_n)+\operatorname{diam}(K_m)+1 \geq |n-m|$. 
Moreover, if $v\in K_n$ for some $n\in \Z$ then the geodesic connecting $v$ to $X_n$ is contained in $K_n$, so that 
\begin{equation}
\label{eq:diam}|K_n| \geq \operatorname{diam}(K_n)
\end{equation} for every $n\in \Z$. We deduce from \eqref{eq:max} that 
$\bigcup_{n \geq N} K_n$ and $\bigcup_{n \geq N} K_{-n}$ are not adjacent for $N$ sufficiently large almost surely. It follows that the union  
$\bigcup_{n = -N}^N K_n$ is a finite set separating $G$ into two disjoint sets of vertices for sufficiently large $N$ almost surely, so that $G$ is two-ended almost surely as claimed. 

Finally, the fact that $\left(V,\eps d_G(x,y),\rho\right)$ converges to $(\R,|x-y|,0)$ in the pointed Gromov-Hausdorff topology as $\eps \downarrow 0$ follows easily from \eqref{eq:max} and \eqref{eq:diam}. Indeed,  these estimates imply that if $n(v)$ denotes the unique index such that $v \in K_{n(v)}$ for each $v\in V$ then
\[
\limsup_{r\to\infty}\max_{u,v \in B(\rho,Ar)} \frac{1}{r} \bigl|d(u,v)-|n(u)-n(v)|\bigr| \leq \limsup_{r\to\infty}\max_{-Ar \leq n \leq Ar}  \frac{2}{r}\operatorname{diam}(K_n)=0
\] 
for every $A \geq 1$. This implies that the correspondence
\[
 \Bigl\{ \bigl(v, -A \vee r^{-1} n(v) \wedge A\bigr) : v\in B_G(\rho,Ar) \Bigr\} \cup \Bigl\{ \bigl(X_{\lfloor rx \rfloor},x\bigr) : x \in [-A,A] \Bigr\}
\]
between $(B_G(\rho,Ar), r^{-1} d_G,\rho)$ and $([-A,A],d_\R,0)$ has distortion tending to zero as $r\to\infty$ for every $A \geq 1$, which implies the claim.
 \qedhere

\end{proof}

\begin{remark}
With a little more work one can show that every two-ended unimodular random rooted graph has linear volume growth and converges to $\R$ under rescaling almost surely. This is should be compared to the results of \cite{MR3573922}.
\end{remark}

We now deduce Theorem \ref{thm:GH} from Theorem \ref{thm:BStwoends}; this will be an exercise in  compactness.






\begin{proof}[Proof of Theorem \ref{thm:GH}] For each $\eps >0$ define $\mathcal{E}_\eps:\mathcal{G}_\bullet \to [0,1]$ by
\begin{multline*}
\mathcal{E}_\eps (g,v) = d^\mathrm{loc}_\mathrm{GH}\biggl(\Bigl(V,\, \eps d_{g},\, v \Bigr),\, \Bigl(\R,\, d_\R,\, 0 \Bigr)\biggr) \\ =  \sum_{k \geq 1} 2^{-k} d_\mathrm{GH}\biggl(\Bigl(B_g(v,k \eps^{-1} ),\, \eps d_g,\, v\,\Bigr),\, \Bigl([-k,k],\, d_\R,\, 0\Bigr)\biggr).
\end{multline*}
Since the Gromov-Hausdorff distance between two compact, metric spaces is always bounded by their diameter (consider the correspondence $\mathscr{R}=X \times Y$), the sum defining $\mathcal{E}_\eps$ converges uniformly and, since each summand is clearly continuous, we deduce that $\mathcal{E}_\eps$ is continuous on $\mathcal{G}_\bullet$ for each $\eps>0$. It follows that $\inf_{\eps>0} \mathcal{E}_\eps$ is upper semi-continuous, and in particular that the set  $\{ (g,v) : \inf_{\eps>0}\mathcal{E}_\eps(g,v) < x\}$ is open in $\mathcal{G}_\bullet$ for each $x>0$.
%
Applying the portmanteau theorem, we deduce that for each $x>0$ the map $\mu \mapsto \mu(\inf_{\eps>0}\mathcal{E}_\eps(G,\rho) < x)$ is a lower semi-continuous function with respect to the weak topology on the space of probability measures $\mathcal{P}(\mathcal{G}_\bullet)$. 

Let $\rho_n$ be a uniform root vertex of $G_n$ and let $\mu_n$ be the law of $(G_n,\rho_n)$. As discussed in the introduction, the uniform integrability assumption ensures that the set $A=\{ \mu_n : n \geq 1\}$ is a precompact subset of $\mathcal{P}(\mathcal{G}_\bullet)$ \cite[Proposition 21 and Exercise 26]{CurienNotes} and hence also sequentially precompact since $\mathcal{P}(\mathcal{G}_\bullet)$ is metrisable. Moreover, the closure $\overline{A}$ is contained in the set of unimodular probability measures on $\mathcal{G}_\bullet$. 
 Since $|V_n| \to \infty$, it follows that the set $\overline{A} \setminus A$ coincides with the set of limits of subsequences of $(\mu_n)_{n\geq 1}$. 
Thus, Theorem \ref{thm:BStwoends} implies that
\[
\mu\Bigl(\inf_{\eps>0}\mathcal{E}_\eps(G,\rho) < x \Bigr) \geq C^{-1}
\]
for every $\mu \in \overline{A} \setminus A$ and $x>0$. It therefore follows by a standard compactness argument using lower semi-continuity that 
\[
\liminf_{n\to\infty} \mu_n\left(\inf_{\eps>0}\mathcal{E}_\eps(G,\rho)<x\right) \geq C^{-1}
\]
for every $x>0$, which is equivalent to the claim. Indeed, if this were not the case then 
 there would exist $x>0$ and a subsequence $\sigma(n)$ such that $\lim_{n\to\infty} \mu_{\sigma(n)}\left(\inf_{\eps>0}\mathcal{E}_\eps(G_{\sigma(n)},\rho_{\sigma(n)}) < x\right) < C^{-1}$. Using sequential compactness, we would be able to take a further subsequence $\tau$ such that $\lim_{n\to\infty} \mu_{\tau(n)}\left(\inf_{\eps>0}\mathcal{E}_\eps(G_{\tau(n)},\rho_{\tau(n)})<x\right) < C^{-1}$ and that $\mu_{\tau(n)}$ converges to some $\mu \in \overline{A} \setminus A$. But then lower semi-continuity would give that \[C^{-1} > \lim_{n\to\infty} \mu_{\tau(n)}\left(\inf_{\eps>0}\mathcal{E}_\eps(G_{\tau(n)},\rho_{\tau(n)})<x\right) \geq \mu\left(\inf_{\eps>0}\mathcal{E}_\eps(G,\rho)<x\right) \geq C^{-1},\] a contradiction.
\end{proof}

\subsection*{Acknowledgments} We thank Jonathan Hermon and Matthew Tointon for helpful comments on a draft. We also thank the anonymous referee for their helpful suggestions.

  \bibliographystyle{abbrv}
  \bibliography{unimodularthesis.bib}

\begin{thebibliography}{1}

\bibitem{AL07}
David Aldous and Russell Lyons.
\newblock Processes on unimodular random networks.
\newblock {\em Electron. J. Probab.}, 12:no. 54, 1454--1508, 2007.

\bibitem{BFT}
Itai Benjamini, Hilary Finucane, and Romain Tessera.
\newblock On the scaling limit of finite vertex transitive graphs with large
  diameter.
\newblock {\em Combinatorica}, 37(3):333--374, 2017.

\bibitem{BeSc}
Itai Benjamini and Oded Schramm.
\newblock Recurrence of distributional limits of finite planar graphs.
\newblock {\em Electron. J. Probab.}, 6:no. 23, 13 pp. (electronic), 2001.

\bibitem{BI}
Dmitri Burago, Yuri Burago, and Sergei Ivanov.
\newblock {\em A course in metric geometry}, volume~33 of {\em Graduate Studies
  in Mathematics}.
\newblock American Mathematical Society, Providence, RI, 2001.

\bibitem{CurienNotes}
Nicolas Curien.
\newblock Random graphs: the local convergence point of view.
\newblock 2017.
\newblock Unpublished lecture notes. Available at
  \url{https://www.math.u-psud.fr/~curien/cours/cours-RG-V3.pdf}.

\bibitem{Haggstrom97}
Olle H\"aggstr\"om.
\newblock Infinite clusters in dependent automorphism invariant percolation on
  trees.
\newblock {\em Ann. Probab.}, 25(3):1423--1436, 1997.

\bibitem{Klenkebook}
Achim Klenke.
\newblock {\em Probability theory}.
\newblock Universitext. Springer-Verlag London, Ltd., London, 2008.
\newblock A comprehensive course, Translated from the 2006 German original.

\bibitem{tessera2017scaling}
Romain Tessera and Matthew Tointon.
\newblock Scaling limits of cayley graphs with polynomially growing balls.
\newblock {\em arXiv preprint arXiv:1711.08295}, 2017.

\bibitem{MR3573922}
Matthew C.~H. Tointon and Ariel Yadin.
\newblock Horofunctions on graphs of linear growth.
\newblock {\em C. R. Math. Acad. Sci. Paris}, 354(12):1151--1154, 2016.

\end{thebibliography}







\end{document}